\documentclass[12pt,reqno]{amsart}
\usepackage{amsmath,amsthm,amsfonts,amssymb,amscd,  multicol, verbatim, enumerate, stackrel}
\usepackage{verbatim}
\usepackage{fullpage}
\usepackage{xcolor}
\usepackage[english]{babel}
\usepackage{cite}
\usepackage{amsfonts}
\usepackage{latexsym}
\usepackage{amsthm}
\usepackage{amsopn}
\usepackage{amsmath}
\usepackage[all]{xy}
\usepackage{subcaption}
\usepackage{hyperref}

\usepackage{tikz}
\usetikzlibrary{arrows}

\makeatletter
\@namedef{subjclassname@2020}{%
  \textup{2020} Mathematics Subject Classification}
\makeatother

\usepackage{verbatim}
\usepackage{amssymb}
\usepackage{mathrsfs}
\usepackage{float}
\usepackage{fullpage}
\usepackage{amscd}
\usepackage{amscd}
\usepackage{color}

\usepackage{graphicx}

\theoremstyle{plain}
\newtheorem*{thm*}{Theorem}
\newtheorem*{con*}{Conjecture}
\newtheorem{thm}{Theorem}[section]
\newtheorem{pro}[thm]{Proposition}

\newtheorem{lem}[thm]{Lemma}

\newtheorem{cor}[thm]{Corollary}

\theoremstyle{definition}
\newtheorem{defn}[thm]{Definition}

\theoremstyle{remark}
\newtheorem{rmk}[thm]{Remark}

\title{ Grapes and Alexander duality}
\author{Mario Marietti}

\date{}

\address{Mario Marietti, Dipartimento  di Ingegneria Industriale e Scienze Matematiche, Universit\`a Politecnica delle Marche, via Brecce Bianche, 60131 Ancona,  Italy}
\email{m.marietti@univpm.it}


\begin{document}

\maketitle

\begin{abstract}
In this paper, we prove that the property of being a grape (in any of its variants) is invariant under Alexander duality. The explicitly determined (simple-)homotopy type of a grape can be transferred to its Alexander dual via Combinatorial Alexander Duality in (co)homology. We also provide several applications.
\end{abstract}
\section{Introduction}

Recursive techniques play a central role in topological and algebraic combinatorics. Many important classes of simplicial complexes, such as shellable, vertex decomposable, or nonevasive complexes, are defined by requiring that suitable recursive decompositions preserve good combinatorial properties. These notions have proved extremely powerful, but they are sometimes too restrictive to apply to a wide range of naturally occurring examples.

Topological and combinatorial grapes were introduced in \cite{MT1} and \cite{MT2} as a flexible framework to study the homotopy classes of simplicial complexes through controlled vertex-based decompositions. Roughly speaking, a grape is a simplicial complex that can be recursively reduced by choosing vertices whose deletion and link satisfy suitable structural or topological conditions. Thanks to its recursive predictability, the grape structure provides a systematic combinatorial certificate to prove that a complex is contractible or homotopy equivalent to a wedge of spheres of possibly different dimensions.  
This recursive structure has proved particularly effective in the study of complexes arising from combinatorial structures with natural distinguished objects that induce vertices in the associated simplicial complexes, such as forests with their leaves (see \cite{MT2}).

In this paper, we give four purely combinatorial variations on the theme of grape, each of which is a topological grape 
and hence contractible or homotopy equivalent to a wedge of spheres (Definition~\ref{grappoli}). 
Depending on the variant, the condition is:
\begin{itemize}
\item the link and deletion must themselves be grapes, and
\item there is some extra condition that controls how the topology behaves when the pieces are glued back.
\end{itemize}
We establish a compatibility between all these recursive structures of grapes and Alexander duality (Theorem~\ref{grappolo-chiuso}). Specifically, we show that if a simplicial complex is a grape (in any of its four variants), then its Alexander dual is also a grape of the same type. This is not true for many standard notions. The symmetry within the theory simplifies the topological analysis of dual structures.

Furthermore, we show that a simplicial complex that happens to fall under the strongest version of grape is either simple-homotopy equivalent to the void complex or to the boundary of a cross-polytope of a some dimension $n$  (Theorem~\ref{strongcombgrapesim}). Our method also provides a recursive procedure to compute the dimension $n$. This, together with the Combinatorial Alexander Duality, can be used to transfer simple-homotopy information between a simplicial complex and its dual (Corollary~\ref{AlexDual}). Recall that simple-homotopy equivalence is finer than homotopy equivalence.

We also provide several applications of the above results, showing that a number of simplicial complexes in the literature fall naturally under the theory discussed above (see Section~\ref{appli}).

\section{Preliminaries}
In this section we recall the basic definitions concerning simplicial complexes and fix the notation to be used throughout the paper.

\begin{defn}
A \emph{simplicial complex} is a pair $(\Delta, X)$, where $X$ is a finite set and $\Delta$ is a family of subsets of $X$ such that, whenever $\sigma \in \Delta$ and $\sigma' \subseteq \sigma$, we have $\sigma' \in \Delta$.
\end{defn}

For brevity, we sometimes say that $\Delta$ is the simplicial complex, omitting the dependence on $X$, which is called the \emph{ground set} of the simplicial complex. The elements of $\Delta$ are the \emph{faces} of the simplicial complex, and the maximal faces are called \emph{facets}. The element of a face of cardinality $1$ is called a  \emph{vertex}. We do not require that $\{x\} \in \Delta $ for all $x \in X$, that is, not every element of $X$ is necessarily a vertex of the simplicial complex.

Note that, for each ground set $X$, there are two simplicial complexes with no vertices:  the \emph{irrelevant complex}, which we call also the \emph{$(-1)$-dimensional sphere} or the \emph{boundary of the $0$-dimensional cross-polytope} and has only one face, the empty face,  and the \emph{void complex}, which we call also the \emph{$(-1)$-dimensional simplex} and has no faces.  Further examples of simplicial complexes on a nonempty ground set $X$ are the $(|X|-1)$-dimensional simplex $2^X$, that is, the set of all subsets of $X$, its boundary $2^X\setminus \{X\}$, and the boundary of the $n$-dimensional cross-polytope, denoted $\partial \beta_n$. 

Every simplicial complex $(\Delta, X) $, other than the irrelevant complex, admits  a standard geometric realization.  Let $W$ be the real vector space having $X$ as a basis.  The realization of $\Delta $ is the union of the convex hulls of its faces.  Whenever we refer to a topological 
property of $\Delta $, we implicitly mean a property of this geometric realization, endowed with the topology induced by the Euclidean topology of $W$. By convention, the void complex  is regarded as contractible.

Let $x \in X$. The \emph{deletion} and  the \emph{link}  of $x$ are the simplicial complexes on the ground set $X\setminus \{x\}$ defined, respectively, as follows:
\begin{eqnarray*} 
\operatorname{dl}_{x}(\Delta) & = & \bigl\{ \sigma \in \Delta :  x \notin  \sigma \bigr\}, \\
\operatorname{lk}_{x}(\Delta) & = & \bigl\{ \sigma \in \Delta : x\notin \sigma \text{ and }  \sigma \cup \{x\} \in \Delta \bigr\}.
\end{eqnarray*}
If $\Delta _1 , \ldots , \Delta_k$ are simplicial complexes with ground set $X$, we define 
$$
\Delta_1 * \cdots * \Delta_k =
\bigl\{ m_1 \cup \cdots \cup m_k : m_i \in \Delta_i \bigr\},
$$
and we call it the \emph{join} of   $\Delta _1 , \ldots , \Delta_k$.
If $x,y \in X$, we set 
\begin{eqnarray*} 
A_x \bigl( \Delta \bigr) & = & \Delta * \{\emptyset , x\}, \\
\Sigma_{x,y} \bigl( \Delta \bigr) & = & \Delta * \{\emptyset , x,y\}. 
\end{eqnarray*}
Both $A_x \bigl( \Delta \bigr)$ and $\Sigma_{x,y} \bigl( \Delta \bigr)$ 
are simplicial complexes.
If $x \neq y$  and no face of $\Delta $ contains either of them, then 
$A_x \bigl( \Delta \bigr)$ and $\Sigma_{x,y} \bigl( \Delta \bigr)$ are called, respectively, 
the \emph{cone on $\Delta$ with apex $x$} and the \emph{suspension of $\Delta$}. We denote them by $\operatorname{cone}_x \bigl( \Delta \bigr)$ and $\operatorname{susp}_{x,y} \bigl( \Delta \bigr)$.
If  $x \neq y$, $x' \neq y'$, and none of $x,y,x',y'$ is contained in any face of 
$\Delta $, then the suspensions 
$\operatorname{susp}_{x,y} \bigl( \Delta \bigr)$ and $\operatorname{susp}_{x',y'} \bigl( \Delta \bigr)$ are 
isomorphic. Hence, in this case, we sometimes drop the subscripts from the notation.
It is well-known that 
if $\Delta $ is contractible, then $\operatorname{susp} (\Delta )$ is contractible, and that 
if $\Delta $ is homotopy equivalent to a $k$-dimensional sphere, then $\operatorname{susp} (\Delta )$ is homotopy equivalent  to a $(k+1)$-dimensional sphere.  

For ease of reference, we state explicitly the following straightforward fact.
\begin{pro}
\label{precontra}
For every simplicial complex $(\Delta, X)$ and every $x\in X$, we have 
\begin{eqnarray*} 
 \operatorname{cone}_x (\operatorname{lk}_{x}(\Delta)) \cup  \operatorname{dl}_{x}(\Delta) &=& \Delta, \\
\operatorname{cone}_x (\operatorname{lk}_{x}(\Delta)) \cap  \operatorname{dl}_{x}(\Delta) &=& \operatorname{lk}_{x}(\Delta).
\end{eqnarray*}
\end{pro}

\section{Grapes}
Topological and combinatorial grapes were first explored in \cite{MT1} (although they were not referred to by this name at the time) and were later formally introduced in \cite{MT2}. Here, we recall their definitions, together with some variations on the theme.
Recall that  a subcomplex  $\Delta '$ of $\Delta $ is  \emph{contractible in $\Delta$} if the inclusion map $\Delta' \hookrightarrow \Delta$ is homotopic to a constant map.

\begin{defn} \label{grappoli}
A simplicial complex $(\Delta,X)$ is a
\begin{itemize}
\item \emph{topological grape} if 
\begin{enumerate}
\item there exists $a \in X$ such that $\operatorname{lk}_{a}(\Delta)$ is contractible in $\operatorname{dl}_{a}(\Delta)$ and 
both $\operatorname{dl}_{a}(\Delta)$ and  $\operatorname{lk}_{a}(\Delta)$ are topological grapes, or 
\item $\Delta $ is contractible or $\Delta = \{ \emptyset \}$;
\end{enumerate}

\item \emph{combinatorial grape} if 
\begin{enumerate}
\item there exists $a \in X$ such that $\operatorname{lk}_{a}(\Delta)$ is contained in a cone that is itself contained in $\operatorname{dl}_{a}(\Delta)$, 
and both $\operatorname{dl}_{a}(\Delta)$ and  $\operatorname{lk}_{a}(\Delta)$ are combinatorial grapes, or 
\item $\Delta $ has at most one vertex;
\end{enumerate}

\item \emph{strong combinatorial grape} if 
\begin{enumerate}
\item there exists $a \in X$ such that at least one of $\operatorname{lk}_{a}(\Delta)$ and $\operatorname{dl}_{a}(\Delta)$ is a cone, 
and both $\operatorname{dl}_{a}(\Delta)$ and  $\operatorname{lk}_{a}(\Delta)$ are strong combinatorial grapes, or 
\item $\Delta $ has at most one vertex;
\end{enumerate}

\item \emph{weak combinatorial grape} if 
\begin{enumerate}
\item there exists $a \in X$ and a simplicial complex $\Gamma$ such that
\[
\operatorname{lk}_{a}(\Delta) \subset \Gamma \subset \operatorname{dl}_{a}(\Delta),
\]
$\Gamma$ is simple-homotopy equivalent to the void complex, and both
$\operatorname{dl}_{a}(\Delta)$ and $\operatorname{lk}_{a}(\Delta)$ are weak combinatorial grapes, or 
\item $\Delta $ has at most one vertex;
\end{enumerate}

\item \emph{strong-weak combinatorial grape} if 
\begin{enumerate}
\item there exists $a \in X$ such that at least one of $\operatorname{lk}_{a}(\Delta)$ and $\operatorname{dl}_{a}(\Delta)$ is simple-homotopy equivalent to the void complex,  
and both $\operatorname{dl}_{a}(\Delta)$ and  $\operatorname{lk}_{a}(\Delta)$ are strong-weak combinatorial grapes, or 
\item $\Delta $ has at most one vertex.
\end{enumerate}
\end{itemize}
\end{defn}

\begin{rmk}
The following hierarchies are immediate:
\begin{enumerate}
\item weak combinatorial grape $\implies$ topological grape;
\item strong combinatorial grape $\implies$ combinatorial grape $\implies$ weak combinatorial grape;
\item strong combinatorial grape $\implies$ strong-weak combinatorial grape $\implies$ weak combinatorial grape.
\end{enumerate}
\end{rmk}

\begin{rmk}
Instances of strong combinatorial grapes already appear in \cite{MT1} and \cite{MT2}. However, the terminology \lq\lq strong\rq\rq is borrowed from \cite{GKL}, where these complexes are simply called strong grapes.
\end{rmk}

\begin{rmk}
The simplicial complex whose facets are $\{a,b\}$, $\{b,c\}$, $\{c,d\}$, $\{d,e\}$, and $\{e,a\}$ (a \lq\lq $5$-cycle\rq\rq) is a weak combinatorial grape but not a combinatorial grape. This is the only simplicial complex with fewer than six vertices that is not a combinatorial grape.
\end{rmk}

\begin{rmk}
Disjoint unions of topological (respectively, combinatorial or weak combinatorial) grapes are again topological (respectively, combinatorial or weak combinatorial) grapes, but this is property no longer holds for strong combinatorial grapes and strong-weak combinatorial grapes. 
\end{rmk}

\begin{rmk}
Several properties of simplicial complexes—such as non-evasiveness, vertex decomposability, shellability, and pure shellability—bear a formal resemblance to the notion of a grape. Nevertheless, a grape does not in general satisfy any of these properties.
\end{rmk}

Arguing as in  \cite[Proposition 3.3]{MT2}, we have the following generalization of it.  
\begin{pro}
\label{ognicomp}
Let $(\Delta, X)$ be a simplicial complex and let $a \in X$. If $\operatorname{lk}_{a}(\Delta)$ is contractible in  $\operatorname{dl}_{a}(\Delta)$, then $\Delta$ is homotopy equivalent to  $\operatorname{dl}_{a}(\Delta) \vee \Sigma \operatorname{lk}_{a}(\Delta)$.  Therefore:
\begin{itemize}
\item each connected component of a topological grape (and hence also of a weak combinatorial grape or a combinatorial grape) is either contractible or homotopy equivalent to a wedge of spheres;
\item a strong-weak combinatorial grape (and hence also a strong combinatorial grape) is either contractible or homotopy equivalent to a sphere.
\end{itemize}
\end{pro}

Being a grape depends only on the vertex set and not on the ground set.  
\begin{lem}
\label{togli uno}
Let $(\Delta,X)$ be a simplicial complex and let $x$ be an element of $X$ that is not a vertex of $\Delta$.
Then $(\Delta,X)$ is a grape of any type if and only if $(\Delta,X\setminus{x})$ is a grape of the same type.
\end{lem}
\begin{proof}
We use induction on $|X|$, the case $|X|\leq 1$ being trivial.
Since $(\Delta,X)$ and $(\Delta,X\setminus{x})$ have the same number of vertices, we may suppose that they have more than one vertex.

Let us fix a type of grape. For brevity, in the rest of this proof, by a grape we mean a grape of this fixed type.

Suppose that $(\Delta,X)$ is a grape and let us show that also $(\Delta,X \setminus{x})$ is a grape.
There exists $a \in X$ such that both $\operatorname{lk}{a}(\Delta,X)$ and $\operatorname{dl}{a}(\Delta,X)$ are grapes and, moreover, they satisfy a further condition that, although different according to the type of grape, does not depend on the ground set but only on the set of faces of the simplicial complexes involved.

If $a= x$, then
$\operatorname{dl}_{a}(\Delta,X)=(\Delta,X\setminus{x})$, and hence $(\Delta,X\setminus{x})$ is a grape.

If $a\neq x$, then $\operatorname{lk}{a}(\Delta,X)$ and $\operatorname{lk}{a}(\Delta,X\setminus{x})$ have the same faces but different ground sets: the former has ground set $X\setminus{a}$ and the latter has ground set $X\setminus{x,a}$. Analogously for $\operatorname{dl}{a}(\Delta,X)$ and $\operatorname{dl}{a}(\Delta,X\setminus{x})$. By the induction hypothesis, both $\operatorname{lk}{a}(\Delta,X\setminus{x})$ and $\operatorname{dl}{a}(\Delta,X\setminus{x})$ are grapes. Moreover, they satisfy the further condition since $\operatorname{lk}{a}(\Delta,X)$ and $\operatorname{dl}{a}(\Delta,X)$ do. Thus $(\Delta,X\setminus{x})$ is a grape.

Conversely, suppose that $(\Delta,X\setminus{x})$ is a grape and let us show that also $(\Delta,X)$ is a grape.
Since $\operatorname{lk}{x}(\Delta,X)$ is the void complex and
$\operatorname{dl}{x}(\Delta,X)=(\Delta,X\setminus{x})$, we have that both $\operatorname{lk}{x}(\Delta,X)$ and
$\operatorname{dl}{x}(\Delta,X)$ are grapes. Furthermore, since $\operatorname{lk}_{x}(\Delta,X)$ is the void complex, the extra condition is satisfied. Hence, $(\Delta,X)$ is a grape.
\end{proof}

We end this section with the following easy lemma, which is needed later.

\begin{lem}
\label{cone}
\leavevmode
\begin{enumerate}
\item
\label{cone1} A cone is a strong combinatorial grape.
\item 
\label{cone2} The simplicial complex $ 2^X \setminus \{X\}$ on the ground set $X$ is a strong combinatorial grape.
\end{enumerate}
\end{lem}
\begin{proof}
(\ref{cone1}). Let $\Delta $ be a simplicial complex with ground set $X$ that is cone with apex $x$. We use induction on $|X|$, the case $|X|\leq 1$ being trivial.
Let $a\in X\setminus\{x\}$.  Then $\operatorname{lk}_{a}(\Delta,X)$ and  $\operatorname{dl}_{a}(\Delta,X)$ are both cones with apex $x$, and they are strong combinatorial grapes by the induction hypothesis.

(\ref{cone2}). We use induction on $|X|$, the case $|X|\leq 1$ being trivial. Let $\Delta = 2^X \setminus \{X\}$ and let $a$ be any element in $X$.   Then $\operatorname{lk}_{a}(\Delta,X) = 2^{X\setminus \{a\}} \setminus (X\setminus \{a\})$, which is a strong combinatorial grape by the induction hypothesis,  and  $\operatorname{dl}_{a}(\Delta,X)= 2^{X\setminus \{a\}}$, which is a strong combinatorial grape and a cone.
\end{proof}

\section{Simple-homotopy}

For the reader's convenience, we recall the basic notions of simple-homotopy theory, which gives
a combinatorial framework that refines classical homotopy theory for simplicial complexes (see \cite{C}).  
Let $\sigma$ be a maximal face of a simplicial complex 
$\Delta $ and $\tau \subset \sigma $ with  $|\tau | = | \sigma | -1$. 
If $\sigma $ is the only face of $\Delta $ properly containing $\tau $, then the removal 
of $\sigma $ and $\tau $ is called an \emph{elementary collapse}. The inverse of this procedure is an \emph{elementary expansion}. 
Two simplicial complexes are \emph{simple-homotopy equivalent}  if they are related by a sequence of  elementary collapses and  elementary  expansions. 
This is a finer equivalence relation than homotopy equivalence, i.e., if two simplicial complexes are simple-homotopy equivalent, then they are homotopy equivalent.

We say that \emph{$\Delta $ collapses to $\Delta '$} 
if there is a sequence of elementary collapses leading from $\Delta $ to $\Delta '$. 
A simplicial complex is  \emph{collapsible} if  it collapses to the void complex (or, equivalently, to a simplicial complex with only one vertex as it is customary to say mirroring the definition of contractible).  For example, every cone is collapsible.
Again, it is easy to check that collapsibility implies contractibility, but the converse is not true. The smallest
example of a contractible simplicial complex that is not collapsible is the dunce hat (see \cite{Z}). 
For ease of reference, we state explicitly the following straightforward fact.
\begin{lem}
\label{secollassaanchesospensione}
Let $K$ and $K'$ be simplicial complexes.  If $K$ and $K'$ are simple-homotopy equivalent, then  $\operatorname{susp}(K)$ and  $\operatorname{susp}(K')$ are simple-homotopy equivalent.
\end{lem}

The following result is needed in the proof of  Theorem~\ref{strongcombgrapesim}.
\begin{pro}
\label{simple-hom}
Let $(\Delta, X)$ be a simplicial complex and $a \in X$.
\begin{enumerate}
\item
\label{simple-hom1}
 If $\operatorname{dl}_{a}(\Delta)$ is a cone, then $\Delta$ is simple-homotopy equivalent to   $\operatorname{susp}(\operatorname{lk}_{a}(\Delta))$.
\item
\label{simple-hom2}
 If $\operatorname{lk}_{a}(\Delta)$  is collapsible, then $\Delta$ collapses to  $\operatorname{dl}_{a}(\Delta)$.
\end{enumerate}
\end{pro}
\begin{proof}
(\ref{simple-hom1}). If $\Delta$ is a cone with apex $a$, then $\operatorname{dl}_{a}(\Delta)=\operatorname{lk}_{a}(\Delta)$ and the suspension of a cone is a cone. If $\Delta$ is a cone with apex $v$ for some $v$ different from $a$, then both  $\operatorname{dl}_{a}(\Delta)$ and $\operatorname{lk}_{a}(\Delta)$ are cones with apex $v$. In particular,  $\operatorname{lk}_{a}(\Delta)$ is collapsible, as well as  its suspension by Lemma~\ref{secollassaanchesospensione}. Hence, we may suppose that  $\Delta$ is not a cone, and in this case  the assertion is given by \cite[Theorem~3.5]{MT1}

(\ref{simple-hom2}). Consider a sequence of elementary collapses from $\operatorname{lk}_{a}(\Delta)$ to the void complex: $(F_1, F_1\setminus\{x_1\}),  (F_2, F_2\setminus\{x_2\}), \ldots, (F_s, F_s\setminus\{x_s\})$.  Thus, the sequence  $(F_1\cup \{a\}, (F_1\cup \{a\})\setminus\{x_1\} ),  (F_2\cup \{a\}, (F_2\cup \{a\})\setminus\{x_2\}), \ldots, (F_s\cup \{a\}, (F_s\cup \{a\})\setminus\{x_s\})$ shows that   $Cone_a(\operatorname{lk}_{a}(\Delta))$ collapses to  $\operatorname{lk}_{a}(\Delta)$.  By Proposition~\ref{precontra}, the same sequence shows also that $\Delta$ collapses to  $\operatorname{dl}_{a}(\Delta)$.
\end{proof}

The following result and its proof provide a tool to study the simple-homotopy type of strong combinatorial grapes.
\begin{thm}
\label{strongcombgrapesim}
A strong combinatorial grape $\Delta$ is simple-homotopy equivalent to either the void complex  or  the boundary of a cross-polytope.
\end{thm}
\begin{proof}
We use induction on the number of vertices of $\Delta$. The result is trivial if  $\Delta $ has at most one vertex. If  $\Delta $ has more than one vertex, then by the definition of a strong combinatorial grape there exists $a \in X$ such that at least one of $\operatorname{lk}_{a}(\Delta)$ and $\operatorname{dl}_{a}(\Delta)$ is a cone, 
and both $\operatorname{dl}_{a}(\Delta)$ and  $\operatorname{lk}_{a}(\Delta)$ are strong combinatorial grapes.

Suppose $\operatorname{dl}_{a}(\Delta)$ is a cone. Then $\Delta$ is simple-homotopy equivalent to   $\operatorname{susp}(\operatorname{lk}_{a}(\Delta))$ by Proposition~\ref{simple-hom}. Being a strong combinatorial grape, $\operatorname{lk}_{a}(\Delta)$  is simple-homotopy equivalent to either the void complex  or the boundary of a cross-polytope, by the induction hypothesis. The assertion follows by  Lemma~\ref{secollassaanchesospensione} since the suspension of the void complex is the void complex and the suspension of the boundary of the $n$-dimensional cross-polytope is  the boundary of the $(n+1)$-dimensional cross-polytope.

Suppose $\operatorname{lk}_{a}(\Delta)$ is a cone. Then $\Delta$ collapses  to  $\operatorname{dl}_{a}(\Delta)$ by Proposition~\ref{simple-hom}. Being a strong combinatorial grape, $\operatorname{dl}_{a}(\Delta)$  is simple-homotopy equivalent to either the void complex  or the boundary of a cross-polytope, by the induction hypothesis. 
\end{proof}
\begin{rmk}
The proof of Theorem~\ref{strongcombgrapesim} gives a procedure to detect when a strong combinatorial grape  is simple-homotopy equivalent to the void complex and when to  the boundary of a cross-polytope, and, in the latter case, the dimension of that cross-polytope.
\end{rmk}

\section{Grapes and Alexander duals}
In this section, we show that the property of being a grape, in any of the variants introduced in the previous section, is invariant under Alexander duality.

Let $(\Delta,X)$ be a simplicial complex. The \emph{Alexander dual} of $(\Delta,X)$ is the simplicial complex $(\Delta^*,X)$  defined by
$$\Delta^* = \{F : X \setminus F \notin \Delta\}.$$

Note that $\Delta$ and $\Delta^*$ have the same ground set and that the set of faces $\Delta^*$ of the Alexander dual of $(\Delta,X)$ depends not only on $\Delta$ but also on $X$.

For ease of reference, in the following proposition we collect some straightforward facts that are needed in the proof of Theorem~\ref{grappolo-chiuso}. 
\begin{pro}
\label{duale_cono}
Let $(\Delta,X)$ be a simplicial complex.
\begin{enumerate}
\item
\label{dcone1} If $(\Gamma,X)$ is a subcomplex of $(\Delta,X)$, then $(\Delta^*,X)$ is a subcomplex of $(\Gamma^*,X)$.
\item 
\label{dcone2} If $(\Delta,X)$ is a cone with apex $x$ then also $(\Delta^*, X)$ is a cone with apex $x$.
\item 
\label{dcone3} If $(\Delta,X)$ is simple-homotopy equivalent to the void complex,  then also $(\Delta^*, X)$ is simple-homotopy equivalent to the void complex.
\item
\label{dl-lk-lk-dl}
 Let $a\in X$. Then
$(\operatorname{dl}_{a}(\Delta,X))^*= \operatorname{lk}_{a}(\Delta^*,X)$
and
  $(\operatorname{lk}_{a}(\Delta,X))^*= \operatorname{dl}_{a}(\Delta^*,X)$.
\end{enumerate}
\end{pro}

Recall that both $\operatorname{dl}_{a}(\Delta,X)$ and $\operatorname{lk}_{a}(\Delta,X)$ have  $X\setminus\{a\}$ as ground set.

\begin{thm}
\label{grappolo-chiuso}
The properties of being a combinatorial grape, a strong combinatorial grape, a weak combinatorial grape, or a strong-weak combinatorial grape are invariant under Alexander duality.
\end{thm}
\begin{proof}
Recall that being a strong combinatorial grape is the strongest of the properties in the statement.

Let $(\Delta,X)$ be a grape in any of its variants. Let us prove that  $\Delta^*$ is a grape of the same type. We use induction on $|X|$, the case $|X|\leq 1$ being trivial. Let $|X| > 1$. If  $\Delta$ is the void complex, then $\Delta^* = 2^X$, which is a strong combinatorial grape.  If $\Delta$ is the irrelevant complex, then $\Delta^* = 2^X \setminus X$, which is a strong combinatorial grape by  Lemma~\ref{cone}, (\ref{cone2}). If $\Delta$ has exactly one vertex, then  $\Delta^*$ is a cone by  Proposition~\ref{duale_cono}, (\ref{dcone2}) and a strong combinatorial grape by  Lemma~\ref{cone}, (\ref{cone1}).

For the rest of the proof, we need different arguments for the different types of grape.

\emph{Argument for combinatorial grapes}. 
Let $a$ in $X$ be such that $\operatorname{lk}_{a}(\Delta,X)$ is contained in a cone $C$ contained in $\operatorname{dl}_{a}(\Delta,X)$  and both $\operatorname{lk}_{a}(\Delta,X)$ and  $\operatorname{dl}_{a}(\Delta,X)$ are combinatorial grapes. 
By Proposition~\ref{duale_cono},  (\ref{dcone1}) and (\ref{dcone2}), the simplicial complex  $C^*$ is a cone  and  satisfies 
$$  (\operatorname{dl}_{a}(\Delta,X))^*  \subseteq C^*  \subseteq  (\operatorname{lk}_{a}(\Delta,X))^*,$$ 
which becomes 
$$  \operatorname{lk}_{a}(\Delta^*,X) \subseteq C^*  \subseteq  \operatorname{dl}_{a}(\Delta^*,X),$$ 
by Proposition~\ref{duale_cono}, (\ref{dl-lk-lk-dl}). Since $\operatorname{dl}_{a}(\Delta,X)$ and $ \operatorname{lk}_{a}(\Delta,X)$ are combinatorial grapes with ground set $X\setminus \{a\}$, their duals also are combinatorial grapes by the induction hypothesis. Thus $  \operatorname{lk}_{a}(\Delta^*,X)$ and $ \operatorname{dl}_{a}(\Delta^*,X)$ are combinatorial grapes, and the assertion is proved.

\emph{Argument for strong combinatorial grapes}. 
Let $a$ in $X$ be such that  at least one of $\operatorname{lk}_{a}(\Delta)$ and $\operatorname{dl}_{a}(\Delta)$ is a cone 
and both $\operatorname{dl}_{a}(\Delta)$ and  $\operatorname{lk}_{a}(\Delta)$ are strong combinatorial grapes.
 By Proposition~\ref{duale_cono}, (\ref{dl-lk-lk-dl}), we have 
$\operatorname{lk}_{a}(\Delta^*,X)= (\operatorname{dl}_{a}(\Delta,X))^*$ and $  \operatorname{dl}_{a}(\Delta^*,X)=  (\operatorname{lk}_{a}(\Delta,X))^*$. Hence,   at least one of $\operatorname{lk}_{a}(\Delta^*,X)$ and $  \operatorname{dl}_{a}(\Delta^*,X)$ is a cone by  Proposition~\ref{duale_cono},  (\ref{dcone2}).
Since $\operatorname{dl}_{a}(\Delta,X)$ and $ \operatorname{lk}_{a}(\Delta,X)$ are strong combinatorial grapes with ground set $X\setminus \{a\}$, their duals also are strong combinatorial grapes by the induction hypothesis. 
Thus $  \operatorname{lk}_{a}(\Delta^*,X)$ and $ \operatorname{dl}_{a}(\Delta^*,X)$ are strong combinatorial grapes, and the assertion is proved.

\emph{Argument for weak combinatorial grapes}. 
Let $a$ in $X$ be such that both $\operatorname{dl}_{a}(\Delta)$ and  $\operatorname{lk}_{a}(\Delta)$ are weak combinatorial grapes, and let $\Gamma$ be a simplicial complex that is simple-homotopy equivalent to the void complex and satifies  $\operatorname{lk}_{a}(\Delta) \subset \Gamma \subset \operatorname{dl}_{a}(\Delta)$.

By Proposition~\ref{duale_cono},  (\ref{dcone1}) and (\ref{dcone3}), the simplicial complex  $\Gamma^*$ is  simple-homotopic to the void complex  and  satisfies 
$$  (\operatorname{dl}_{a}(\Delta,X))^*  \subseteq \Gamma^*  \subseteq  (\operatorname{lk}_{a}(\Delta,X))^*,$$ 
which becomes 
$$  \operatorname{lk}_{a}(\Delta^*,X) \subseteq \Gamma^*  \subseteq  \operatorname{dl}_{a}(\Delta^*,X),$$ 
by Proposition~\ref{duale_cono}, (\ref{dl-lk-lk-dl}). Since $\operatorname{dl}_{a}(\Delta,X)$ and $ \operatorname{lk}_{a}(\Delta,X)$ are weak combinatorial grapes with ground set $X\setminus \{a\}$, their duals also are weak combinatorial grapes by the induction hypothesis. Thus $  \operatorname{lk}_{a}(\Delta^*,X)$ and $ \operatorname{dl}_{a}(\Delta^*,X)$ are weak combinatorial grapes, and the assertion is proved.

\emph{Argument for strong-weak combinatorial grapes}. 
Let $a$ in $X$ be such that  at least one of $\operatorname{lk}_{a}(\Delta)$ and $\operatorname{dl}_{a}(\Delta)$ is simple-homotopic to the void complex  
and both $\operatorname{dl}_{a}(\Delta)$ and  $\operatorname{lk}_{a}(\Delta)$ are strong-weak combinatorial grapes.
 By Proposition~\ref{duale_cono}, (\ref{dl-lk-lk-dl}), we have 
$\operatorname{lk}_{a}(\Delta^*,X)= (\operatorname{dl}_{a}(\Delta,X))^*$ and $  \operatorname{dl}_{a}(\Delta^*,X)=  (\operatorname{lk}_{a}(\Delta,X))^*$. Hence,   at least one of $\operatorname{lk}_{a}(\Delta^*,X)$ and $  \operatorname{dl}_{a}(\Delta^*,X)$ is simple-homotopic to the void complex by  Proposition~\ref{duale_cono},  (\ref{dcone3}).
Since $\operatorname{dl}_{a}(\Delta,X)$ and $ \operatorname{lk}_{a}(\Delta,X)$ are strong-weak combinatorial grapes with ground set $X\setminus \{a\}$, their duals also are strong-weak combinatorial grapes by the induction hypothesis. Thus $  \operatorname{lk}_{a}(\Delta^*,X)$ and $ \operatorname{dl}_{a}(\Delta^*,X)$ are strong-weak combinatorial grapes, and the assertion is proved.
\end{proof}

\begin{rmk}
Note that the Alexander dual version of Lemma~\ref{togli uno} is the following trivial statement:
if $(\Delta, X)$ is a simplicial complex and $x$ in $X$ is such that $ \operatorname{dl}_{x}(\Delta,X) = 2^{X\setminus \{x\}}$, then we have that $\Delta$ is  grape of any type  if and only if $\operatorname{lk}_{x}(\Delta,X)$  is a  grape of the same type.
Indeed,  $x $ is not a vertex of a simplicial complex $\Gamma$ if and only if $X\setminus \{x\} \in \Gamma^*$, i.e., if and only if  $ \operatorname{dl}_{x}(\Gamma^*,X) = 2^{X\setminus \{x\}}$. In that case,  $ \operatorname{lk}_{x}(\Gamma^*,X)$ is the dual of $\Gamma$ as a simplicial complex with ground set $X\setminus  \{x\}$ .
\end{rmk}

The following is a corollary of  Theorems~\ref{strongcombgrapesim} and~\ref{grappolo-chiuso}. 
Recall that $\partial \beta_n$ denotes the boundary of the $n$-dimensional cross-polytope, that is,
$$
\partial \beta_n = S^0 * S^0 * \dots * S^0 \quad (n \text{ times}),
$$
which gives a triangulation of the $(n-1)$-dimensional sphere.

\begin{cor}
\label{AlexDual}
Let  $(\Delta,X)$ be a strong combinatorial grape with $|X|>0$. 
\begin{enumerate}
\item 
If $\Delta$  is simple-homotopy equivalent to the void complex, then $\Delta^*$ is a strong combinatorial grape which is simple-homotopy equivalent to the void complex.
\item  If $\Delta$ is simple-homotopy equivalent to  $\partial \beta_n$, then $\Delta^*$ is a strong combinatorial grape which is simple-homotopy equivalent to  $\partial \beta_{(|X|-n-1)}$.
\end{enumerate}
\end{cor}
\begin{proof}
The statements follow from Theorems~\ref{strongcombgrapesim} and~\ref{grappolo-chiuso}, together with the Combinatorial Alexander Duality, which 
states that, for any simplicial complex $(\Gamma,X)$ and any abelian group $G$, there are isomorphisms $\widetilde{H}_i(\Gamma;G) \cong \widetilde{H}^{|X|-i-3}(\Gamma^*;G)$ and $\widetilde{H}^i(\Gamma;G) \cong \widetilde{H}_{|X|-i-3}(\Gamma^*;G)$ where $\widetilde{H}$ denotes reduced simplicial (co)homology (see, for example, \cite[Theorem~2.1]{BCP}).
\end{proof}

\section{Applications}
\label{appli}
In this section, we present several applications of the results from the previous section, all concerning simplicial complexes arising from graphs, which have been extensively studied. We believe that further applications could be developed in the future.

\subsection{Simplicial complexes associated with forests}
In \cite{MT2}, a few simplicial complexes are shown to be grapes and are then shown to belong to certain homotopy classes. Thanks to the results in the previous sections, we can directly prove that  also their duals are grapes and find their simple-homotopy classes. 

In the following list, we fix our notation on a finite simple graph $G=(V,E)$ (undirected graph with no loops or multiple edges and so we may think of an edge also as a subset of $V$ of cardinality 2):
\smallskip 

$
\begin{array}{@{\hskip-1.3pt}l@{\qquad}l}
V &  \textrm{the vertex set of $G$}, 
\\
E &  \textrm{the edge set of $G$}, 
\\
N(S)& \textrm{the open  neighborhood of $S$, i.e. } \bigl\{ w \in V :  \exists s \in S \textrm{ with } \{ s , w \} \in E \bigr\},  
\\
N[S]& \textrm{the closed  neighborhood of $S$, i.e. } \bigl\{ w \in V :  \exists s \in S \textrm{ with } \{ s , w \} \in E \bigr\} \cup S,  
\\
\gamma (G) &  \textrm{the domination number of $G$, i.e. } \min \bigl\{ |D| : D \textrm{ is a dominating set of $G$}  \bigr\},
\\ 

i (G) &  \textrm{the independent domination number of $G$, i.e., } 
\\
 & \min \bigl\{ |D| : D \textrm{ is an independent  dominating set of $G$}  \bigr\},
\\ 

\alpha _0 (G) &  \textrm{the vertex covering  number of $G$, i.e., } \min \bigl\{ |C| :  C \text{ is a vertex cover of $G$} \bigr\},
\\ 
\beta _1 (G) &  \textrm{the matching  number of $G$, i.e., } \max \bigl\{ |M| :  M \text{ is a matching of $G$} \bigr\},
\\ 
(\operatorname{Ind}(G),V) &  \textrm{the independence complex of $G$ ($F$ is a face  if it is independent)}, 
\\
(\operatorname{Dom}(G),V) &  \textrm{the dominance complex of $G$ ($F$ is a face  if $V\setminus F$ is a dominating set},  
\\
 & \textrm{i.e., $ (V\setminus F) \cap N[\{v\}] \neq \emptyset$ for all $v$ in $V$)}, 
\\
(\operatorname{EC}(G),E) &  \textrm{the edge covering complex of $G$ ($F$ is a face  if $E\setminus F$ is an edge cover},  
\\
& \textrm{i.e.,  for every $v$ in $V$ there exists $e$ in $E\setminus F$ with $v\in e$)}, 
\\
(\operatorname{ED}(G),E) &  \textrm{the edge dominance complex of $G$ ($F$ is a face  if $E\setminus F$ is a dominating set of}  
\\
& \textrm{the line dual of $G$,  i.e.,  for every $e$ in $E$ there exists $e'$ in $E\setminus F$ with $e' \cap e \neq \emptyset$)}, 
\end{array}$
\bigskip

Recall that a subset 
$D$ of  $V$ is \emph{dominating} if $N[D] = V$,  \emph{independent} if no two vertices in $D$ are adjacent (i.e., $|D \cap e| \leq 1$ for all $e$ in $E$), and is  a   \emph{vertex cover of $G$} if every edge of $G$ contains a vertex of $D$.  If $D$ is not a vertex cover of $G$, then we say that $D$ is a \emph{noncover} of $G$. An \emph{edge cover of $G$} is a subset $S$ of  $E$ such that the union of all the endpoints of the edges in $S$ is $V$.  
A \emph{matching of $G$} is a subset $M$ of $E$ of pairwise disjoint edges.  The \emph{line dual}   of $G$, denoted $\operatorname{ld}(G)$, is the graph whose vertex set is $E$ and where $\{e_1, e_2\}$ is an edge if $e_1 \neq e_2$ and  $e_1 \cap e_2 \neq \emptyset$.

We refer the reader to \cite{D} for all undefined notation on graph theory.

The results in \cite{MT2} concerning the independent complex $\operatorname{Ind}(G)$,  the dominance complex $ \operatorname{Dom}(G) $,  the edge covering complex $\operatorname{EC}(G)$, and  the edge dominance complex $\operatorname{ED}(G)$ for a forest $G$ are formulated in terms of combinatorial grapes and homotopy types. In this paper, we reformulate them in terms of strong combinatorial grapes and simple-homotopy types.
\begin{thm}
\label{MMMTTT}
Let $G$ be a forest. Then $\operatorname{Ind}(G)$, $ \operatorname{Dom}(G) $, $\operatorname{EC}(G)$ and $\operatorname{ED}(G)$ are all strong combinatorial grapes. Moreover 
\begin{enumerate}
\item $\operatorname{Ind}(G)$ is simple-homotopy equivalent to either the void complex 
or $\partial \beta_{i(G)}$, and in the latter case $i(G)=\gamma(G)$;
\item  $ \operatorname{Dom}(G) $ is simple-homotopy equivalent to $\partial \beta_{\alpha_0(G)}$;
\item $\operatorname{EC}(G)$ is simple-homotopy equivalent to either the void complex  or $\partial \beta_{(|E|-|V|+i(G))}$, and in the latter case $i(G)=\gamma(G)$;
\item $\operatorname{ED}(G)$ is simple-homotopy equivalent to $\partial \beta_{(|E|-\alpha_0(G))}$.
\end{enumerate}
\end{thm}
\begin{proof}
The  proofs  in  \cite{MT2} of the fact that these simplicial complexes are combinatorial grapes in fact show that these complexes are strong combinatorial grapes. Hence, we may apply Theorem~\ref{strongcombgrapesim}.
\end{proof}
\begin{rmk}
Recall that, by a well-known result of K\"onig,  we have $\alpha _0 (T) = \beta _1 (T)$ for all bipartite graphs $T$ (see~\cite[Theorem~2.1.1]{D}). So, in particular, this equality holds for forests.
 \end{rmk}

Let us describe the faces of the Alexander duals of the simplicial complexes described above:
\begin{itemize}
\item $F$ is a face of  $\operatorname{Ind}(G)^*$ if and only if $F$ is a noncover of $G$,
\item  $F$ is a face of  $\operatorname{Dom}(G)^*$ if and only if $F$ is not dominating,
\item  $F$ is a face of  $\operatorname{EC}(G)^*$ if and only if $F$ is not an edge cover,
\item  $F$ is a face of  $\operatorname{ED}(G)^*$ if and only if there exists an edge in $E$ that is disjoint from every edge in $F$.
%$\big( \cup_{e\in F} e\big)$ fails to meet all edges of $G$
\end{itemize} 

\begin{rmk}
The complex $\operatorname{Ind}(G)^*$ is called the \emph{noncover complex} of $G$ in \cite{KK}.
\end{rmk}

\begin{rmk}
The complex $\operatorname{Dom}(G)^*$ is Lov\'asz’s neighborhood complex of the complement $\overline{G}$ of $G$. The neighborhood complex of a graph $T = (V_T, E_T)$ was introduced by Lov\'asz in connection with Kneser’s conjecture in \cite{L} and is the simplicial complex on $V_T$ whose faces are the subsets $F$ of  $V_T$ having a common neighbor; that is, $F$ is a face if and only if there exists $v \in V_T$ such that $N(v) \supseteq F$ (note that since there are no loops, a vertex is not a neighbor of itself). The complement $\overline{T}$ of a graph $T$ is the graph on the same vertex set as $T$ in which two vertices are adjacent if and only if they are not adjacent in $T$.
\end{rmk}
\begin{rmk}
The complex $\operatorname{ED}(G)^*$ coincides with $\operatorname{Dom}(\operatorname{ld}(G))^*$, the Alexander dual of the dominance complex of the line dual of $G$.
\end{rmk}

\begin{cor}
\label{eccoecco}
Let $G=(V,E)$ be a forest. Then $\operatorname{Ind}(G)^*$, $ \operatorname{Dom}(G) ^*$, $\operatorname{EC}(G)^*$ and $\operatorname{ED}(G)^*$ are all strong combinatorial grapes. Moreover 
\begin{enumerate}
\item $\operatorname{Ind}(G)^*$ is simple-homotopy equivalent to either the void complex  or $\partial \beta_{(|V|-i(G)-1)}$, and in the latter case $i(G)=\gamma(G)$;
\item  $ \operatorname{Dom}(G)^* $ is simple-homotopy equivalent to $\partial \beta_{(|V|-\alpha_0(G)-1)}$;
\item $\operatorname{EC}(G)^*$ is simple-homotopy equivalent to either the void complex  or $\partial \beta_{(|V|-i(G)-1)}$, and in the latter case $i(G)=\gamma(G)$;
\item $\operatorname{ED}(G)^*$ is simple-homotopy equivalent to $\partial \beta_{(\alpha_0(G)-1)}$.
\end{enumerate}
\end{cor}
\begin{proof}
It follows by Theorem~\ref{grappolo-chiuso}, Corollary~\ref{AlexDual}, and Theorem~\ref{MMMTTT}.
\end{proof}
\begin{rmk}
Corollary~\ref{eccoecco} explains the numerology in Theorem~\ref{MMMTTT} for  $\operatorname{EC}(G)$ and  $\operatorname{ED}(G)$, that are arguably Alexander duals of more natural simplicial complexes.
\end{rmk}

\begin{rmk}
By \cite[Theorem~4.16]{MT2}, the forests $G$ such that $\operatorname{Ind}(G)$ and $\operatorname{Ind}(G)^*$ are simple-homotopy equivalent to the void complex are exactly those such that $\operatorname{EC}(G)$ and $\operatorname{EC}(G)^*$ are simple-homotopy equivalent to the void complex. Similarly, the forests $G$ such that $\operatorname{Ind}(G)$ and $\operatorname{Ind}(G)^*$ are simple-homotopy equivalent to the boundary of a cross-polytope are exactly those such that $\operatorname{EC}(G)$ and $\operatorname{EC}(G)^*$ are simple-homotopy equivalent to the boundary of a cross-polytope. 
In \cite{MT3}, the former are called \emph{conical} and the latter \emph{spherical}. Characterizations of conical and spherical forests in terms of dominating sets, independent dominating sets, and edge covers are given in \cite[Theorem~4.10]{MT3} (see also \cite[Theorem~4.9]{MT3}).
\end{rmk}

\subsection{The path-missing and path-free complexes}
In \cite{GKL}, the authors introduce two simplicial complexes attached to a directed graph with two distinguished vertices, which they call the \emph{path-missing complex} and \emph{path-free complex}, and study their homotopy classes. These two simplicial complexes are  Alexander dual to each other. In this subsection, we show that such complexes follow under the theory of grapes and so we may apply the results in the previous sections to obtain their simple-homotopy classes. 

Let $G = (V,E)$ be a directed multigraph (different arcs sharing both source and target are allowed, as well as loops), where  $V$ is its vertex set and  $E$ its arc set. 
We let $V'_G$ be the set of nonsinks, i.e., 
$V'_G=\{v\in V : \text{there exists an arc whose source is $v$}\}$.
A \emph{walk} from $v_0$ to $v_n$ is a sequence $(v_0, e_1, v_1, e_2, v_2, \ldots, e_n, v_n)$, where $n$ is a nonnegative integer, $v_0, v_1, \ldots, v_n \in V$,  $e_1, e_2, \ldots, e_n \in E$, and $e_i$ has source $v_{i-1}$ and target $v_i$, for $i\in [n]$.
A \emph{path} is a walk that does not visit any vertex more than once. A \emph{cycle} is a nontrivial walk such that the first and last vertices coincide.
We say that a subset $F$ of $E$ \emph{contains} a path $p$ if $F$ contains each arc of $p$. 

Let  $s$ and $t$ be two (possibly coinciding) distinguished vertices of $G$.  
The \emph{path-free complex} $\mathcal{PF}(G,s,t)$ and the \emph{path-missing complex} $\mathcal{PM}(G,s,t)$ are simplicial complexes on the ground set $E$ defined as follows:
    \begin{align*}
      \mathcal{PF}(G,s,t) &= \{ F  : F \text{ contains no path from $s$ to $t$}\} ; \\
\mathcal{PM}(G,s,t)&=      \{ F  : E\setminus F \text{ contains a path from $s$ to $t$ }\} .
  \end{align*}
Note that, when $|E|=0$, we have that:
\begin{itemize}
\item  if $s\neq t$, then $ \mathcal{PF}(G,s,t)$ is the irrelevant complex and  $\mathcal{PM}(G,s,t)$ is the void complex;
\item  if $s=t$, then $ \mathcal{PF}(G,s,t)$ is the void complex and  $\mathcal{PM}(G,s,t)$ is the irrelevant complex. 
\end{itemize}

In the following lemma, we collect the results in Lemmas~3.6 and~3.10 of \cite{GKL}, which are needed in the proof of Theorem~\ref{path_free_and_missing}. An arc $e$ in $E$ is  \emph{$G_{s,t}$-useless} (or simply $G$-useless or useless if no confusion arises)  if no path from $s$ to $t$ contains $e$. In particular, any loop is useless.

\begin{lem}
\label{GKLresults}
Let $G = (V,E)$ be a directed multigraph.
\begin{enumerate}
\item
\label{GKL1}
Let $e \in E$. Then:
\begin{enumerate}
\item[(a)] $\operatorname{dl}_{e}(\mathcal{PF}(G,s,t)) = \mathcal{PF}(G \backslash e,s,t)$;
\item[(b)] If the source of $e$ is $s$, then $\operatorname{lk}_{e}(\mathcal{PF}(G,s,t)) = \mathcal{PF}(G / e,s,t)$.	
\end{enumerate}
\item
\label{GKL2}  Let $e$ be a non-useless arc with source $s$. If $|E| > 1$ and there are no other arcs with the same target as $e$, then  $G\backslash e$ has a useless arc.
\end{enumerate}
\end{lem}

The following theorem summarizes the main results in \cite{GKL}, with the difference that the statements in \cite{GKL} concern homotopy classes rather than simple-homotopy classes. The fact that these complexes are strong combinatorial grapes is proved in \cite[Proposition~3.9]{GKL}; we reprove it here since the proof naturally goes along with the study of the dimension of the associated cross-polytope. In \cite{GKL}, the dimension of the associated sphere is obtained via discrete Morse theory rather than via grape analysis, as remarked in \cite[Remark~9.4]{GKL}. With Theorem~\ref{path_free_and_missing}, we also settle this problem. Furthermore, thanks to Theorem~\ref{grappolo-chiuso} and Corollary~\ref{AlexDual}, one half of the proof of Theorem~\ref{path_free_and_missing} follows from the other half.

\begin{thm}
\label{path_free_and_missing}
Let $G = (V,E)$ be a directed multigraph with $|E|>0$, and $s,t\in V$.  Then the following statements hold.
\begin{enumerate}
\item 
\label{teo1}
The path-free complex  $\mathcal{PF}(G,s,t)$ is a strong combinatorial grape which is: 
\begin{enumerate}
\item[(a)] simple-homotopy equivalent to the void complex, if $G$ contains a useless arc or a cycle;
 \item[(b)] simple-homotopy equivalent to $\partial \beta_{(|V'_G| -1)}$, otherwise.
 \end{enumerate}
\item
\label{teo2}
The path-missing complex $\mathcal{PM}(G,s,t)$  is a strong combinatorial grape which is: 
\begin{enumerate}
\item[(a)] simple-homotopy equivalent to the void complex, if $G$ contains a useless arc or a cycle;
\item[(b)] simple-homotopy equivalent to $\partial \beta_{(|E| - |V'_G| )}$, otherwise.
\end{enumerate}
\end{enumerate}
\end{thm}
\begin{proof}
By Theorem~\ref{grappolo-chiuso} and Corollary~\ref{AlexDual}, the assertions (\ref{teo1}) and (\ref{teo2}) are equivalent. 

Let us prove only (\ref{teo1}). We use induction on $|E|$. If $|E|=1$, then the assertion is easy to check. 

 If $G$ has a useless arc $e$, then $\mathcal{PF}(G,s,t)$ is a cone with apex $e$, and hence it is a strong combinatorial grape by Lemma~\ref{cone},~(\ref{cone1}) and is simple-homotopy equivalent to the void complex. 
Furthermore, note that each arc is useless when $s=t$ as well as  when $s\neq t$ and there are no arcs with $s$  as source. 

Hence, we may suppose $|E|>1$, $s\neq t$, that no arc of $G$ is useless and that there is at least one arc $e$ whose source is $s$.

By the induction hypothesis, both  $\operatorname{dl}_{e}(\mathcal{PF}(G,s,t))$ and  $\operatorname{lk}_{e}(\mathcal{PF}(G,s,t)) $ are strong combinatorial grapes since $\operatorname{dl}_{e}(\mathcal{PF}(G,s,t)) = \mathcal{PF}(G \backslash e,s,t)$ and $\operatorname{lk}_{e}(\mathcal{PF}(G,s,t)) = \mathcal{PF}(G / e,s,t)$ by Lemma~\ref{GKLresults},~(\ref{GKL1}).

\bigskip
 
Let $s'$ be the target of $e$.

\emph{Argument for when $G$ has no arcs other than $e$ with target $s'$.} In this case, $G \backslash e$ has a useless arc by Lemma~\ref{GKLresults},~(\ref{GKL2}), hence 
$\operatorname{dl}_{e}(\mathcal{PF}(G,s,t))$ is a cone and  $\mathcal{PF}(G,s,t)$ is a strong combinatorial grape which is simple-homotopic to $\operatorname{susp}( \mathcal{PF}(G / e,s,t)) $. By Lemma~\ref{secollassaanchesospensione}, we need to show that the following are equivalent:
\begin{enumerate}
\item $G / e$ has neither useless arcs nor cycles,
\item $G$ has no cycles (we already know that $G$ has no useless arcs),
\end{enumerate}
and that, if the two above conditions are fulfilled, then $V'_{G}=V'_{G/e}+1$. 

It is clear that, if  $G$ has a cycle, then  $G / e$ has a cycle. 

Now suppose that $G/e$ has a cycle. A priori,  such cycle could give rise to either a cycle of $G$, or a path from $s$ to $s'$,  or a path from $s'$ to $s$.  But the case of a path  from $s$ to $s'$ is impossible since $e$ is the only arc of $G$ with target $s'$,  as well as the case of a path from $s'$ to $s$ since an arc whose target is $s$ is clearly useless.
Now suppose that $G/e$ has no cycles and let us show that $G/e$ has no useless arc as well. Toward a contradiction, let $u$ be a $G/e$-useless arc. Then $u$ is contained in a path $p$ of $G$ from $s$ to $t$ since $G$ has no useless arcs. After removing eventually $e$, the path $p$ gives rise to a walk from $s$ to $t$ in $G/e$ containing $u$ which is actually a path since $G / e$ has no cycles. Hence $u$ is not useless also as an arc of $G/e$. Hence, we have shown that the two above conditions are equivalent.

Suppose that the two above conditions are fulfilled. Clearly $s \in V'_{G}$ since $s$ is the source of $e$. If $s'=t$, then no arc in $E\setminus \{e\}$ has source in $\{s,s'\}$ since $G/e$ has no useless arc. Hence $s' \notin (V'_{G} \cup V'_{G/e})$ and $V'_{G}=V'_{G/e} \cup\{s\}$.  If $s' \neq t$, then there must be an arc of $G$ with source $s'$ otherwise $e$ would be $G$-useless. Hence $s' \in V'_{G}$ and $s=s' \in V'_{G/e}$. Hence $|V'_{G}|=|V'_{G/e}|+1$.  

\emph{Argument for when $G$ has an arc $\bar{e}$, with $\bar{e}\neq e$, whose  target is  $s'$.}   Since its target is $s'=s$,  the arc $\bar{e}$ is $G/e$-useless, hence 
$\operatorname{lk}_{e}(\mathcal{PF}(G,s,t))$ is a cone and  $\mathcal{PF}(G,s,t)$ is a strong combinatorial grape which is simple-homotopic to $ \mathcal{PF}(G  \backslash  e,s,t) $. We need to show that the following are equivalent:
\begin{enumerate}
\item $G  \backslash  e$  has neither useless arcs nor cycles,
\item $G$ has no cycles (we already know that $G$ has no useless arcs),
\end{enumerate}
and that, if the two above conditions are fulfilled, then $V'_{G}=V'_{G\backslash e}$. 

If  $G$ has a cycle $c$, then  $c$ cannot contain $s$ since $G$ has no useless arc (and an arc whose target is $s$ is useless). Hence $c$ cannot contain $e$ and gives rise to a cycle of $G  \backslash e$. 

If $G \backslash e$ has a cycle $c$, then $c$ gives rise to a cycle of $G$ as well. 
Now suppose that $G \backslash e$ has no cycles and let us show that $G \backslash e$ has no useless arc as well. Toward a contradiction, let $u$ be a $G \backslash e$-useless arc.  Since $u$ is not $G$-useless, $G$ has path $p$ from $s$ to $t$ containing $u$. Since $u$ is $G \backslash e$-useless, $p$ contains also $e$ (as first arc). Since $\bar{e}$ is not $G$-useless, $G$ has  a path $p'$ from $s$ to $t$ containing $\bar{e}$ (and clearly not containing $e$). The first part of $p'$ up to $\bar{e}$ together with $p$ lacking $e$ provide a walk of $G \backslash e$ from $s$ to $t$ containing $u$ and not containing $e$. Since $u$ is  $G \backslash e$-useless, this path contains a cycle, which is a contradiction.

Suppose that the two above conditions are fulfilled. Clearly $s \in V'_{G}$ since $s$ is the source of $e$. But also $s\in V'_{G\backslash e}$ holds  since otherwise there would be no paths from $s$ to $t$ in $G\backslash e$, and hence $G\backslash e$ would have useless arcs. 
Hence, $|V'_{G}|=|V'_{G/e}|$.  
\end{proof}

\begin{rmk}
Some of the technicalities in the proof of Theorem~\ref{path_free_and_missing}, in particular the use of Lemma~\ref{GKLresults} and the splitting of the proof according to whether $G$ has an extra arc with target $s'$ or not, are taken from \cite{GKL}. Here we give a slightly stronger statement and simplify the proof. Moreover, the treatment is carried out entirely within the theory of grapes.
\end{rmk}

Figure~\ref{nouseless} provides an example of a graph $G$ with distinguished vertices $s$ and $t$ having no useless arcs but containing a cycle. Its associated path-free and path-missing complexes are simple-homotopy equivalent to the void complex (in fact, collapsible), but they are not cones. In the figure, we show the path-free complex $\mathcal{PF}(G,s,t)$, which consists of two tetrahedra $\{A,E,C,D\}$ and $\{F,B,C,D\}$ sharing the face $\{C,D\}$, together with the additional faces $\{A,F\}$, $\{A,F,D\}$, $\{B,E\}$, and $\{B,E,C\}$.
\begin{figure}[h!]

\centering

\begin{minipage}{0.48\textwidth}
\centering
% --------- LEFT FIGURE ----------
\begin{tikzpicture}[scale=2.4]

\tikzset{
  vertex/.style={circle,draw,inner sep=1pt,minimum size=2.5mm},
  edge/.style={->,draw=black,thick}
}

\node[vertex] (u) at (-0.7, 1) {};
\node[vertex] (v) at ( 0.7, 1) {};
\node[vertex,label=below:$t$] (t) at ( 0,   0) {};
\node[vertex,label=above:$s$] (s) at ( 0,   2) {};

\path[edge]
  (s) edge node[left] {$A$} (u) %(s) edge[bend left=10] (u)
  (s) edge node[right] {$E$}(v)
  (u) edge[bend left=20] node[above] {$C$}  (v)
  (v) edge[bend left=20] node[below] {$D$}  (u)
  (u) edge node[left] {$B$} (t)
  (v) edge node[right] {$F$} (t);

\end{tikzpicture}

\subcaption*{The graph $G$}
\end{minipage}
\hfill
\begin{minipage}{0.48\textwidth}
\centering
% --------- RIGHT FIGURE ----------

\begin{tikzpicture}[scale=1.2]

\tikzset{
  vertex/.style={circle,draw,inner sep=1pt,minimum size=2.5mm},
  edge/.style={thick},
  face/.style={fill=gray!20,draw=black,thick,opacity=0.7},
  hidden/.style={thick,dashed}
}

\node[vertex,label=above:$D$] (D) at (0, 1.5) {};
\node[vertex,label=below:$C$] (C) at (0,0.5) {};

\node[vertex,label=above left:$A$] (A) at (-2, 3) {};
\node[vertex,label=below left:$E$] (E) at (-2,-1) {};

\node[vertex,label=above right:$F$] (F) at ( 2, 3) {};
\node[vertex,label=below right:$B$] (B) at ( 2,-1) {};

\draw[edge] (D)--(C);
\draw[edge] (D)--(A);
\draw[edge] (A)--(E);
\draw[edge] (E)--(C);

\draw[edge] (A)--(C);
\draw[hidden] (D)--(E);

\draw[edge] (D)--(F);
\draw[edge] (F)--(B);
\draw[edge] (B)--(C);

\draw[edge] (F)--(C);
\draw[hidden] (D)--(B);

\draw[edge] (A)--(F);
\draw[edge] (E)--(B);
\filldraw[face] (A.center)--(F.center)--(D.center)--cycle;

\filldraw[face] (E.center)--(B.center)--(C.center)--cycle;

\filldraw[fill=gray!40,draw=black,thick,opacity=0.6] (A.center)--(E.center)--(C.center)--(D.center)--cycle;
\filldraw[fill=gray!40,draw=black,thick,opacity=0.6] (F.center)--(B.center)--(C.center)--(D.center)--cycle;

\end{tikzpicture}

\subcaption*{$\mathcal{PF}(G,s,t)$}
\end{minipage}

\caption{The graph $G$ and $\mathcal{PF}(G,s,t)$.}
\label{nouseless}
\end{figure}
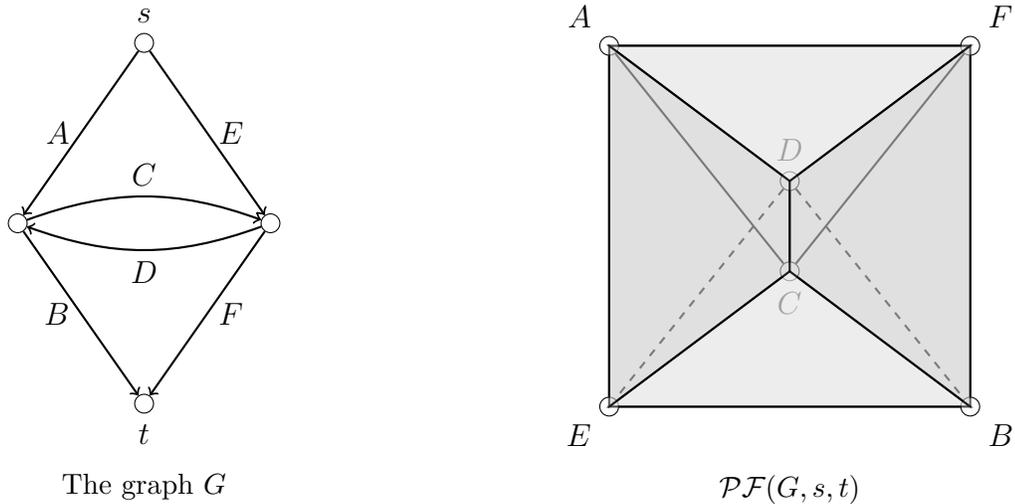

\bigskip
{\bf Acknowledgments:} 
The author is a member of the INDAM group GNSAGA.
% This research did not receive any specific grant from funding agencies in the public, commercial, or not-for-profit sectors.

\end{document}